\documentclass[journal,11pt,onecolumn]{IEEEtran}
\usepackage{amsmath}
\usepackage{amsfonts}
\usepackage{amsthm}
\usepackage{dsfont}
\usepackage{epsfig}
\usepackage{upgreek}
\usepackage{graphicx}

\usepackage{setspace}

\theoremstyle{definition}
\newtheorem{theorem}{Theorem}[section]
\newtheorem{lemma}[theorem]{Lemma}

\newtheorem{definition}[theorem]{Definition}

\newcommand{\Diag}{\mbox{Diag}}
\newcommand{\Abs}{\mbox{Abs}}
\newcommand{\Log}{\mbox{{LogAbs}}}

\newcommand{\Exp}{\mbox{{Exp}}}

\newcommand{\Jn}{\mbox{${\bf J}_n$}}
\newcommand{\Imn}{\mbox{${\bf I}_n$}}

\newcommand{\Jm}{\mbox{${\bf J}_m$}}
\newcommand{\Imm}{\mbox{${\bf I}_m$}}
\newcommand{\Imat}{\mbox{${\bf I}$}}

\newcommand{\Amn}{\mbox{${\bf A}$}}
\newcommand{\Amni}{\mbox{${\bf A}^{\!\mbox{\tiny -1}}$}}

\newcommand{\Lmn}{\mbox{${\bf L}$}}

\newcommand{\Dmp}{\mbox{${\bf D}_{\mbox{\tiny +}}$}}
\newcommand{\Dmu}{\mbox{${\bf D}_{\mbox{\tiny u}}$}}
\newcommand{\Dmv}{\mbox{${\bf D}_{\mbox{\tiny v}}$}}
\newcommand{\Dmpi}{\mbox{${\bf D}_{\tiny +}^{\tiny -1}$}}

\newcommand{\Dmui}{\mbox{${\bf D}_{\tiny u}^*$}}

\newcommand{\Emp}{\mbox{${\bf E}_{\mbox{\tiny +}}$}}
\newcommand{\Empi}{\mbox{${\bf E}_{\tiny +}^{\tiny -1}$}}
\newcommand{\Emu}{\mbox{${\bf E}_{\mbox{\tiny u}}$}}

\newcommand{\Emui}{\mbox{${\bf E}_{\tiny u}^*$}}

\newcommand{\omc}{\mbox{${{\mathds 1}}_m$}}
\newcommand{\omr}{\mbox{${{\mathds 1}}_m^{\textnormal{\tiny{T}}}$}}
\newcommand{\onc}{\mbox{${{\mathds 1}}_n$}}
\newcommand{\onr}{\mbox{${\mathds 1}_n^{\textnormal{\tiny{T}}}$}}

\newcommand{\znc}{\mbox{${{\bf 0}}_n$}}

\newcommand{\dmc}{\mbox{${\bf d}_m$}}

\newcommand{\emc}{\mbox{${\bf e}_m$}}
\newcommand{\emr}{\mbox{${\bf e}_m^{\textnormal{\tiny{T}}}$}}
\newcommand{\enc}{\mbox{${\bf e}_n$}}
\newcommand{\enr}{\mbox{${\bf e}_n^{\mbox{\tiny{T}}}$}}

\newcommand{\umc}{\mbox{${\bf u}_m$}}

\newcommand{\vmc}{\mbox{${\bf v}_m$}}

\newcommand{\vnc}{\mbox{${\bf v}_n$}}
\newcommand{\vnr}{\mbox{${\bf v}_n^{\mbox{\tiny{T}}}$}}
\newcommand{\stimes}{\mbox{$s_{\times}$}}

\newcommand{\Sm}{\mbox{${\bf S}$}}

\newcommand{\uv}{\mbox{${\bf u}$}}
\newcommand{\vv}{\mbox{${\bf v}$}}

\newcommand{\bv}{\mbox{${\bf b}$}}

\newcommand{\Perm}{\mbox{${\bf P}$}}
\newcommand{\Permt}{\mbox{${\bf P}^{\mbox{\tiny{T}}}$}}
\newcommand{\Qm}{\mbox{${\bf Q}$}}

\newcommand{\Rm}{\mbox{${\bf R}$}}

\newcommand{\Vm}{\mbox{${\bf V}$}}
\newcommand{\Vmi}{\mbox{${\bf V}^*$}}
\newcommand{\Xm}{\mbox{${\bf X}$}}
\newcommand{\Xmi}{\mbox{${\bf X}^{\mbox{\tiny -1}}$}}
\newcommand{\Xpinv}{\mbox{${\bf X}^{\!\mbox{\tiny -P}}$}}

\newcommand{\Ym}{\mbox{${\bf Y}$}}
\newcommand{\Ymi}{\mbox{${\bf Y}^{\mbox{\tiny -1}}$}}

\newcommand{\Dm}{\mbox{${\bf D}$}}
\newcommand{\Dmi}{\mbox{${\bf D}^{\mbox{\tiny -1}}$}}
\newcommand{\Dma}{\mbox{${\bf D}_1$}}
\newcommand{\Dmia}{\mbox{${\bf D}_1^{\mbox{\tiny -1}}$}}
\newcommand{\Dmb}{\mbox{${\bf D}_2$}}
\newcommand{\Dmib}{\mbox{${\bf D}_2^{\mbox{\tiny -1}}$}}
\newcommand{\Em}{\mbox{${\bf E}$}}
\newcommand{\Emi}{\mbox{${\bf E}^{\mbox{\tiny -1}}$}}
\newcommand{\Ema}{\mbox{${\bf E}_1$}}
\newcommand{\Emia}{\mbox{${\bf E}_1^{\mbox{\tiny -1}}$}}
\newcommand{\Emb}{\mbox{${\bf E}_2$}}
\newcommand{\Emib}{\mbox{${\bf E}_2^{\mbox{\tiny -1}}$}}

\newcommand{\Fm}{\mbox{${\bf F}$}}
\newcommand{\Fmt}{\mbox{${\bf F}^*$}}
\newcommand{\Gm}{\mbox{${\bf G}$}}
\newcommand{\Gmt}{\mbox{${\bf G}^*$}}

\newcommand{\dinv}[1]{{#1}^{\mbox{\tiny -D}}}
\newcommand{\linv}[1]{{#1}^{\mbox{\tiny -L}}}
\newcommand{\rinv}[1]{{#1}^{\mbox{\tiny -R}}}
\newcommand{\ginv}[1]{{#1}^{\mbox{\tiny -U}}}
\newcommand{\pinv}[1]{{#1}^{\mbox{\tiny -P}}}
\newcommand{\inv}[1]{{#1}^{\mbox{\tiny -1}}}
\newcommand{\DL}{\mbox{${\cal D}_{\mbox{\tiny L}}$}}
\newcommand{\DG}{\mbox{${\cal S}_{\mbox{\tiny U}}$}}
\newcommand{\DGL}{\mbox{${\cal D}_{\mbox{\tiny U}}^{\mbox{\tiny L}}$}}
\newcommand{\DGR}{\mbox{${\cal D}_{\mbox{\tiny U}}^{\mbox{\tiny R}}$}}
\newcommand{\TR}{\mbox{${\bf\cal T}$}}

\newcommand{\norm}[1]{\mbox{$\lVert #1 \rVert$}}

\newcommand{\Um}{\mbox{${\bf U}$}}
\newcommand{\Umi}{\mbox{${\bf U}^*$}}

\newcommand{\xmc}{\mbox{${\bf x}_m$}}

\newcommand{\ync}{\mbox{${\bf y}_n$}}
\newcommand{\ynr}{\mbox{${\bf y}_n^{\mbox{\tiny{T}}}$}}

\newcommand{\Atinv}{\mbox{${\bf A}^{\overset{\sim}{\!\mbox{\tiny -1}}}$}}
\newcommand{\Adinv}{\mbox{${\bf A}^{\!\mbox{\tiny -D}}$}}
\newcommand{\Apinv}{\mbox{${\bf A}^{\!\mbox{\tiny -P}}$}}
\newcommand{\Aginv}{\mbox{${\bf A}^{\!\mbox{\tiny -U}}$}}

\newcommand{\tinv}[1]{\mbox{${#1}^{\overset{\sim}{\mbox{\tiny -1}}}$}}

\newcommand{\yv}{\mbox{$\hat{\bf y}$}}
\newcommand{\av}{\mbox{$\hat{\bf \uptheta}$}}
\newcommand{\yyv}{\mbox{$\hat{\bf y}'$}}
\newcommand{\aav}{\mbox{$\hat{\bf \uptheta}'$}}

\begin{document}

\title{{A Generalized Matrix Inverse that is Consistent\\
           with Respect to Diagonal Transformations} }       
\author{
\IEEEauthorblockN{{\large\bf Jeffrey Uhlmann}}\\
\IEEEauthorblockA{\small University of Missouri-Columbia\\
201 EBW, Columbia, MO 65211\\
Email: uhlmannj@missouri.edu}}
\date{}     
\maketitle
\thispagestyle{empty}

\vspace{-11pt}

\begin{abstract}
A new generalized matrix inverse is derived which is consistent with respect to arbitrary nonsingular diagonal transformations, e.g., it preserves units associated with variables under state space transformations, thus providing a general solution to a longstanding open problem relevant to a wide variety of applications in robotics, tracking, and control systems. The new inverse complements the Drazin inverse (which is consistent with respect to similarity transformations) and the Moore-Penrose inverse (which is consistent with respect to unitary/orthonormal transformations) to complete a trilogy of generalized matrix inverses that exhausts the standard family of analytically-important linear system transformations. Results are generalized to obtain unit-consistent and unit-invariant matrix decompositions and examples of their use are described.\\
~\\
\begin{footnotesize}
\noindent {\bf Keywords}: {Drazin Inverse, Generalized Matrix Inverse, 
Inverse Problems, Linear Estimation, Linear Systems, Machine Learning, Matrix Analysis,  
Moore-Penrose Pseudoinverse, Scale Invariance, Singular Value Decomposition, SVD, 
System Design, Unit Consistency.}
\end{footnotesize}

\end{abstract}

\section{Introduction}

For a nonsingular $n\times n$  matrix\footnote{It is assumed throughout that 
matrices are defined over an associative normed division algebra (real, complex, 
and quaternion). The inverse of a unitary matrix $\Um$ can therefore be
expressed using the conjugate-transpose operator as $\Umi$.}  
$\Amn$ there exists a unique matrix inverse,
$\Amni$, for which certain properties of scalar inverses 
are preserved, e.g., commutativity:
\begin{equation}
      \Amn\Amni ~ = ~ \Amni\Amn ~ = ~ \Imat
\end{equation}
while others have direct analogs, e.g., matrix inversion distributes
over nonsingular multiplicands as:
\begin{equation}
      \inv{(\Xm\Amn\Ym)} ~ = ~ \Ymi\Amni\Xmi 
\end{equation}
When attempting to generalize the notion of a matrix
inverse for singular $\Amn$ it is only possible to 
define an approximate inverse, $\Atinv$, which retains
a subset of the algebraic properties of a true
matrix inverse. For example, 
a generalized inverse definition might simply require
the product $\Amn\Atinv$ to be idempotent in analogy to the identity matrix. 
Alternative definitions might further require:
\begin{equation}
      \Amn\Atinv\Amn ~ = ~ \Amn 
\end{equation}
or
\begin{equation}
      \Atinv\Amn ~ = ~ \Amn\Atinv
\end{equation}
and/or other properties that may be of analytic
or application-specific utility. 

The vast literature\footnote{Much of the literature on generalized inverses is covered
in the comprehensive book by Ben-Israel and Greville~\cite{big}.}
on generalized inverse theory spans more than a century and
can inform the decision about which of the many 
possible generalized inverses is best suited to the needs of a
particular application.
For example, the {\em Drazin} inverse, $\Adinv$, satisfies the following
for any square matrix $\Amn$ and nonsingular matrix 
$\Xm$~\cite{drazin,cmr76,big}:
\begin{eqnarray}
    & ~ &       \Adinv\Amn\Adinv ~ = ~ \Adinv \\
    & ~ &       \Amn\Adinv ~ = ~ \Adinv\Amn \label{dinvcom} \\
    & ~ &       \dinv{(\Xm\Amn\Xmi)} ~ = ~ \Xm\Adinv\Xmi \label{dinvcnst}
\end{eqnarray}
Thus it is applicable when there is need for
commutativity (Eq.(\ref{dinvcom})) and/or consistency with respect
to similarity transformations (Eq.(\ref{dinvcnst})). On the
other hand, the Drazin inverse is only defined for square matrices
and does not guarantee the rank of 
$\Adinv$ to be the same as $\Amn$. Because
the rank of $\Adinv$ may be less than that of $\Amn$ 
(and in fact is zero for all nilpotent matrices), it is
not appropriate for recursive control and estimation
problems (and many other 
applications) that cannot accommodate
progressive rank reduction. 

The Moore-Penrose {\em pseudoinverse}, $\Apinv$, is defined 
for any $m\times n$ matrix $\Amn$ and satisfies conditions
which include the following for any conformant unitary matrices 
$\Um$ and $\Vm$:
\begin{eqnarray}
    & ~ &      \mbox{rank}[\Apinv] ~ = ~ \mbox{rank}[\Amn] \\
    & ~ &      \Amn\Apinv\Amn ~ = ~ \Amn \\
    & ~ &       \Apinv\Amn\Apinv ~ = ~ \Apinv \\
    & ~ &       \pinv{(\Um\Amn\Vm)} ~ = ~ \Vmi\Apinv\Umi \label{udist}
\end{eqnarray}
Its use is therefore appropriate when there is need for
unitary consistency, i.e., as guaranteed by Eq.(\ref{udist}). 
Despite its near-universal use throughout many areas of 
science and engineering ranging from tomography~\cite{berryman00} 
to genomics analysis~\cite{altergolub04}, the Moore-Penrose
inverse is not appropriate for many problems to which it is
commonly applied, e.g., state-space applications 
that require consistency with respect to the choice of units
for state variables. In the case of a square singular transformation
matrix $\Amn$, for example, a simple change of units applied 
to a set of state variables may require an
inverse $\Atinv$ to be preserved under diagonal similarity
\begin{equation}
      \tinv{(\Dm\Amn\Dmi)} ~ = ~ \Dm\Atinv\Dmi
\end{equation}  
where the nonsingular diagonal matrix $\Dm$ defines an arbitrary 
change of units. The Moore-Penrose inverse does not
satisfy this requirement because $\pinv{(\Dm\Amn\Dmi)}$
does not generally equal $\Dm\Apinv\Dmi$. As a concrete
example, given
\begin{equation}
\label{beginexample}
\Dm ~=~ \left[\begin{array}{cc}
                         1 & 0\\ 0 & 2 \end{array}\right] ~ ~ ~ ~
\Amn ~=~ \left[\begin{array}{cc}
                         1/2 & -1/2\\ 1/2 & -1/2 \end{array}\right] ~ ~ ~ ~
\end{equation}
it can be verified that 
\begin{equation}
\Apinv ~=~ \left[\begin{array}{cc}
                         ~~1/2 & ~~1/2\\ -1/2 & -1/2 \end{array}\right] 
\end{equation}
and that 
\begin{equation}
\Dm\Apinv\Dmi ~=~ \left[\begin{array}{cc}
                                         ~1/2 & ~1/4\\ -1 & -1/2 \end{array}\right] 
\end{equation}
which does not equal
\begin{equation}
\label{endexample}
\pinv{(\Dm\Amn\Dmi)} ~=~ \left[\begin{array}{cc}
                                                     ~~0.32 & ~~0.64\\ -0.16 & -0.32 \end{array}\right]. 
\end{equation}

To appreciate the significance of
unit consistency, consider the standard
linear model
\begin{equation}
     \yv ~ = ~ \Amn\cdot\av \label{linmod}
\end{equation}
where the objective is to identify a vector $\av$ of parameter values
satisfying the above equation for a data matrix $\Amn$ and
a known/desired state vector $\yv$. If $\Amn$ is nonsingular
then there exists a unique $\Amni$ which gives the solution
\begin{equation}
     \av ~ = ~ \Amni\cdot\yv
\end{equation}
If, however, $\Amn$ is singular then the
Moore-Penrose inverse could be applied as
\begin{equation}
     \av ~ = ~ \Apinv\cdot\yv
\end{equation}
to obtain a solution. Now 
suppose that $\yv$ and $\av$ are expressed in 
different units as
\begin{eqnarray}
     \yyv & = & \Dm\yv \\
     \aav & = & \Em\av
\end{eqnarray}
where the diagonal matrices $\Dm$ and $\Em$ represent changes
of units, 
e.g.,  from imperial to metric, or rate-of-increase in liters-per-hour 
to a rate-of-decrease in liters-per-minute, or
any other multiplicative change of units. Then Eq.(\ref{linmod}) can 
be rewritten in the new units as
\begin{equation}
     \yyv ~ = ~ \Dm\yv  ~ = ~ (\Dm\Amn\Emi)\cdot\Em\av
     ~ = ~ (\Dm\Amn\Emi)\cdot\aav
\end{equation}
but for which  
\begin{equation}
      \Em\av   ~ \neq ~ \pinv{(\Dm\Amn\Emi)}\cdot\yyv 
\end{equation}
In other words, the change of units applied to the input
does not generally produce the same output in the new units.
This is because the Moore-Penrose inverse only 
guarantees consistency with 
respect to unitary transformations (e.g., rotations) and not with
respect to nonsingular diagonal transformations.
To ensure unit consistency in this example a generalized matrix 
inverse $\Atinv$ would have to satisfy 
\begin{eqnarray}
      \tinv{(\Dm\Amn\Emi)}  & = & \Em\Atinv\Dmi
\end{eqnarray}
Stated more generally, if $\Amn$ represents a mapping
$V\rightarrow W$ from a vector space $V$ to a vector space
$W$ then the inverse transformation $\Atinv$ must preserve
consistency with respect to the application of arbitrary 
changes of units to the coordinates (state variables) associated
with $V$ and $W$. 

Unit consistency (UC) has been suggested in the past
as a critical consideration in specific applications (e.g.,
robotics~\cite{duffy90,doty2} and
data fusion~\cite{jku95}), but the means for
enforcing it have been limited because the most
commonly applied tools in linear systems analysis,
the eigen and singular-value decompositions, are inherently
not unit consistent and therefore require
UC alternatives. This may explain why in practice it is so
common  -- {\em almost reflexive} -- for an arbitrary
criterion such as ``least-squares'' minimization
(which is implicit when the Moore-Penrose inverse
is used) to be applied without consideration for whether it is
appropriate for the application at hand. In this paper
the necessary analytical and practical tools to support
unit consistency are developed. 

The structure of the paper is as follows:
Section~\ref{lrucinv} describes a simple and commonly-used
(at least implicitly)
mechanism for obtaining one-sided unit consistency.
Section~\ref{nzginv} develops a unit-consistent 
generalized inverse for elemental nonzero matrices,
and Section~\ref{gensec} develops the fully general
unit-consistent generalized matrix inverse.
Section~\ref{ucsvd} applies the techniques used
to achieve unit consistency for the generalized inverse 
problem to develop unit-consistent and unit-invariant
alternatives to the singular value decomposition (SVD) 
and other tools from linear algebra.  Finally, Section~\ref{discsec} 
summarizes and discusses the contributions of the paper.

\section{Left and Right UC Generalized Inverses}
\label{lrucinv}

Inverse consistency with respect to a nonsingular left diagonal
transformation, $\linv{(\Dm\Amn)}=\linv{\Amn}\Dmi$, or a right nonsingular diagonal
transformation, $\rinv{(\Amn\Dm)}=\Dmi\rinv{\Amn}$, is straightforward to obtain.
The solution has been exploited implicitly in one form or another in many
applications over the years; however, its formal derivation and
analysis is a useful exercise to establish concepts and notation
that will be used later to derive the fully-general UC solution.

\begin{definition}
\label{defdl}
Given an $m\times n$ matrix $\Amn$, a {\em left diagonal scale function}, 
$\DL[\Amn]\in\mathbb{R}_+^{m\times m}$, is defined as giving
a positive diagonal matrix satisfying the following 
for all conformant positive diagonal matrices $\Dmp$, unitary diagonals $\Dmu$, 
permutations $\Perm$, and unitaries $\Um$:
\begin{eqnarray}
   & ~ & \DL[\Dmp\Amn]\cdot(\Dmp\Amn)   ~ = ~ \DL[\Amn]\cdot\Amn , \label{lscalinv}\\
   & ~ & \DL[\Dmu\Amn] ~ = ~ \DL[\Amn], \label{lduinv}\\
   & ~ & \DL[\Perm\Amn] ~ = ~ \Perm\cdot \DL[\Amn]\cdot\Permt, \label{lperminv}\\
   & ~ & \DL[\Amn\Um] ~ = ~ \DL[\Amn] \label{lruinv} 
\end{eqnarray}
In other words, the product $\DL[\Amn]\cdot\Amn$ is invariant with respect to any positive left-diagonal
scaling of $\Amn$, and $\DL[\Amn]$ is consistent with respect to any left-permutation\footnote{Note: Eq.(27) here
includes a $\Permt$ term errantly missing in the final journal version.} of $\Amn$ and is  invariant 
with respect to left-multiplication by any diagonal unitary and/or 
right-multiplication by any general unitary.
\end{definition}

\begin{lemma}
\label{dconst}
Existence of a left-diagonal scale function according to Definition~\ref{defdl} is established
by instantiating $\DL[\Amn]=\Dm$ with
\begin{equation}
   \Dm(i,i) ~ \doteq ~  
   \begin{cases}
              1/\norm{\Amn(i,:)}  & \norm{\Amn(i,:)}>0\\
             1 & \textnormal{otherwise}
   \end{cases}
\end{equation}
where $\Amn(i,:)$ is row $i$ of $\Amn$ and $\norm{\cdot}$  is a fixed unitary-invariant vector 
norm\footnote{The unitary-invariant norm used here is necessary only because of 
the imposed right-invariant condition of Eq.(\ref{lruinv}).}.
\end{lemma}

\begin{proof}
$\DL[\cdot]$ as defined by Lemma~\ref{dconst} is a
strictly positive diagonal as required, and the left scale-invariance
condition of Eq.(\ref{lscalinv}) holds trivially for any row of $\Amn$ with
all elements equal to zero and holds for every nonzero row $i$ by homogeneity 
for any choice of vector norm as 
\begin{eqnarray}
    \Dmp(i,i)\Amn(i,:) ~ / ~ \norm{\Dmp(i,i)\Amn(i,:)} & = & \Dmp(i,i)\Amn(i,:) ~ / ~ 
                                                             \left(\Dmp(i,i)\cdot\norm{\Amn(i,:)}\right)  \\
   ~ & = & \Amn(i,:) ~ / ~ \norm{\Amn(i,:)}. 
\end{eqnarray}

The left diagonal-unitary-invariance condition of Eq.(\ref{lduinv}) 
is satisfied as $|\Dmu(i,i)|=1$ implies
\begin{equation}
    |(\Dmu\Amn)(i,j)|=|\Amn(i,j)|
\end{equation}
for every element 
$j$ of row $i$ of $\Dmu\Amn$.
The left permutation-invariance of Eq.(\ref{lperminv}) holds as
element $\Dm(i,i)$ is indexed with respect to the rows of $\Amn$,
and the right unitary-invariance
condition of Eq.(\ref{lruinv}) is satisfied by 
the assumed unitary invariance of the vector norm applied
to the rows of $\Amn$. 
\end{proof}

If $\Amn$ has {\em full support}, i.e., no row or column with all
elements equal to zero, then $\DL[\Dmp\Amn]=\Dmpi\cdot\DL[\Amn]$.
If, however, there exists a row $i$ of $\Amn$ with all
elements equal to zero then the $i$th diagonal element of 
$\DL[\Dmp\Amn]$ is $1$ according to Lemma~\ref{dconst},
so the corresponding element of $\Dmpi\cdot\DL[\Amn]$ will
be different unless $\Dmpi(i,i)=1$. Eq.(\ref{lscalinv}) holds 
because such elements are only applied to scale
rows of $\Amn$ with all elements equal to zero. The following
similarly holds in general
\begin{equation}
     \DL[\Dmp\Amn]\cdot\Amn ~ = ~ \Dmpi\cdot\DL[\Amn]\cdot\Amn, \label{dlident}
\end{equation}
and because any row $i$ of zeros in $\Amn$ implies that column $i$ of $\pinv{\Amn}$ will be 
zeros, the following also holds in general
\begin{equation}
     \pinv{\Amn}\cdot\DL[\Dmp\Amn] ~ = ~ \pinv{\Amn}\cdot\Dmpi\cdot\DL[\Amn].
\end{equation}

At this point it is possible to
derive a left generalized inverse of an arbitrary $m \times n$ matrix
$\Amn$, denoted $\linv{\Amn}$, that is consistent with 
respect to multiplication on the left by an arbitrary nonsingular diagonal matrix.

\begin{theorem}
\label{linvt}
For $m\times n$ matrix $\Amn$, the operator 
\begin{equation}
   \linv{\Amn} ~ \doteq ~ \pinv{\left(\DL[\Amn]\cdot\Amn\right)}\cdot\DL[\Amn]
\end{equation}
satisfies for any nonsingular diagonal matrix $\Dm$:
\begin{eqnarray}
   \Amn\linv{\Amn}\Amn & = & \Amn, \\
   \linv{\Amn}\Amn\linv{\Amn} & = & \linv{\Amn}, \\
   \linv{(\Dm\Amn)} & = & \linv{\Amn} \Dmi, \\
   \mbox{rank}[\linv{\Amn}] & = & \mbox{rank}[\Amn] 
\end{eqnarray}
and is therefore a left unit-consistent generalized inverse.
\end{theorem}
\begin{proof}
The first two generalized inverse properties can be established from the corresponding properties
of the Moore-Penrose inverse as: 
\begin{eqnarray}
\Amn\linv{\Amn}\Amn & = & \Amn \cdot \left\{\pinv{\left(\DL[\Amn]\cdot\Amn\right)}\cdot\DL[\Amn]\right\} \cdot \Amn \\
~ & = & \Amn \cdot \left\{\pinv{\left(\DL[\Amn]\cdot\Amn\right)}\cdot \left(\DL[\Amn]\cdot \Amn\right)\right\} \\
~ & = & (\inv{\DL[\Amn]}\cdot\DL[\Amn])\cdot \Amn \cdot \left\{\pinv{\left(\DL[\Amn]\cdot\Amn\right)}\cdot 
              \left(\DL[\Amn]\cdot \Amn\right)\right\} \\
~ & = & \inv{\DL[\Amn]}\cdot \left\{\underline{\left(\DL[\Amn]\cdot\Amn\right) \cdot \pinv{\left(\DL[\Amn]\cdot\Amn\right)}\cdot 
              \left(\DL[\Amn]\cdot \Amn\right)}\right\} \\
~ & = & \inv{\DL[\Amn]}\cdot \left(\DL[\Amn]\cdot \Amn\right)\\
~ & = & \Amn
\end{eqnarray}
and
\begin{eqnarray}
\linv{\Amn}\Amn\linv{\Amn} & = & \left\{\pinv{\left(\DL[\Amn]\cdot\Amn\right)}\cdot\DL[\Amn]\right\}\cdot\Amn\cdot
                                                        \left\{\pinv{\left(\inv{\DL[\Amn]}\Amn\right)}\cdot\DL[\Amn]\right\}  \\
~ & = & \left\{\underline{\pinv{\left(\DL[\Amn]\cdot\Amn\right)}\cdot\left(\DL[\Amn]\cdot\Amn\right)\cdot
                                                        \pinv{\left(\DL[\Amn]\cdot\Amn\right)}}\right\}\cdot \DL[\Amn]  \\
~ & = &  \pinv{\left(\DL[\Amn]\cdot\Amn\right)} \cdot \DL[\Amn] \\
~ & = & \linv{\Amn}
\end{eqnarray}
The left unit-consistency condition $\linv{(\Dm\Amn)}=\linv{\Amn}\Dmi$, for any nonsingular diagonal
matrix $\Dm$, can be established using a polar 
decomposition $\Dm=\Dmp\Dmu$:
\begin{eqnarray}
   \Dmp & = & \Abs[\Dm] \\
   \Dmu & = &  \Dm\Dmpi
\end{eqnarray}
and exploiting unitary-consistency of the Moore-Penrose 
inverse, i.e., $\pinv{(\Um\Amn)}=\pinv{\Amn}\Umi$, and 
commutativity of $\DL[\cdot]$ with other diagonal matrices:

\begin{eqnarray}
\linv{(\Dm\Amn)} & = &  \pinv{\left(\DL[\Dm\Amn]\cdot\Dm\Amn\right)}\cdot\DL[\Dm\Amn] \\
~ & = & \pinv{\left(\DL[\Dmp\Dmu\Amn]\cdot\Dmp\Dmu\Amn\right)}\cdot\DL[\Dmp\Dmu\Amn] \\
~ & = & \pinv{\left(\DL[\Amn]\cdot\Dmpi\cdot\Dmp\Dmu\Amn\right)}\cdot\Dmpi\cdot\DL[\Amn]  \\
~ & = & \pinv{\left(\DL[\Amn]\cdot(\Dmpi\Dmp)\cdot\Dmu\Amn\right)}\cdot\Dmpi\cdot\DL[\Amn]  \\
~ & = & \pinv{\left(\DL[\Amn]\cdot\Dmu\Amn\right)}\cdot\Dmpi\cdot\DL[\Amn] \label{lmpua}\\
~ & = & \pinv{\left(\DL[\Amn]\cdot\Amn\right)}\cdot\Dmui\cdot\Dmpi\cdot\DL[\Amn] \label{lmpub}\\
~ & = & \pinv{\left(\DL[\Amn]\cdot\Amn\right)}\cdot(\Dmui\Dmpi)\cdot\DL[\Amn] \\
~ & = & \pinv{\left(\DL[\Amn]\cdot\Amn\right)}\cdot\Dmi\cdot\DL[\Amn] \\
~ & = & \left\{\pinv{\left(\DL[\Amn]\cdot\Amn\right)}\cdot\DL[\Amn]\right\}\cdot\Dmi\\
~ & = & \linv{\Amn} \Dmi .
\end{eqnarray}
Lastly, the rank-consistency condition, $\mbox{rank}[\linv{\Amn}] =  \mbox{rank}[\Amn]$,
is satisfied as every operation performed according to
Lemma~\ref{dconst} preserves the rank of the original matrix.
In particular, the rank consistency of $\linv{\Amn}$ derives from the fact
that $\mbox{rank}[\Apinv]  =  \mbox{rank}[\Amn]$.
\end{proof}

A right unit-consistent generalized inverse clearly can be 
derived analogously or  in terms
of the already-defined left operator as
\begin{equation}
   \rinv{\Amn} ~ \doteq ~ \left(\linv{(\Amn^{\mbox{\tiny T}})}\right)^{\mbox{\tiny T}}.
\end{equation}
In terms of the linear model of Eq.(\ref{linmod}) for determining values
for parameters $\av$,
\begin{eqnarray}
     \yv & = & \Amn\cdot\av \nonumber \\
       ~ & \Downarrow & ~ \nonumber \\
     \av & = & \Atinv\cdot\yv \nonumber
\end{eqnarray}
the inverse $\Atinv$ could be instantiated with 
{\em either} $\linv{\Amn}$ or $\rinv{\Amn}$ to
provide, respectively, consistency with respect to
the application of a change of units to $\yv$ or
a change of units to $\av$ -- {\em but not both}. 

\section{UC Generalized Inverse for Elemental-Nonzero Matrices}
\label{nzginv}

The derivations of separate left and right UC inverses 
from the previous section cannot
be applied to achieve general unit consistency, i.e., to obtain
a UC generalized inverse
$\Aginv$ which satisfies
\begin{equation}
   \ginv{(\Dm\Amn\Em)} ~ = ~ \Emi  \Aginv \Dmi
\end{equation}
for arbitrary nonsingular diagonals $\Dm$ and $\Em$. However, a 
{\em joint} characterization of the left and right diagonal transformations
can provide a basis for doing so.

\begin{lemma}
\label{daex}
The transformation of an $m\times n$ matrix $\Amn$ as $\Dm\Amn\Em$, with $m\times m$
diagonal $\Dm$ and $n\times n$ diagonal $\Em$, is equivalent to a Hadamard
(elementwise) matrix product $\Xm\circ\Amn$ for some rank-1 matrix $\Xm$.
\end{lemma}

\begin{proof}
Letting $\dmc=\Diag[\Dm]$ and $\enc=\Diag[\Em]$, the matrix product $\Dm\Amn\Em$ can be expressed as
\begin{eqnarray}
   \Dm\Amn\Em & = & (\dmc\onr)\circ\Amn\circ(\omc\emr) \\
   ~                    & = & \left\{(\dmc\onr)\circ(\omc\emr)\right\}\circ\Amn \\ 
   ~                   & = & (\dmc\enr)\circ\Amn
\end{eqnarray}
where $\onr$ is a row vector of $n$ ones and
$\omc$ is a column vector of $m$ ones.
Letting $\Xm=\dmc\enr$ completes the proof.
\end{proof}

\begin{definition}
\label{defdg}
For an $m\times n$ matrix $\Amn$, left and right {\em general-diagonal scale functions} 
$\DGL[\Amn]\in\mathbb{R}_+^{m\times m}$ and
$\DGR[\Amn]\in\mathbb{R}_+^{n\times n}$
are defined as jointly satisfying the following 
for all conformant positive diagonal matrices $\Dmp$ and $\Emp$, unitary diagonals $\Dmu$ and $\Dmv$, 
and permutations $\Perm$ and $\Qm$:
\begin{eqnarray}
   & ~ & \DGL[\Dmp\Amn\Emp]\cdot(\Dmp\Amn\Emp)\cdot\DGR[\Dmp\Amn\Emp]
                      ~ = ~ \DGL[\Amn]\cdot\Amn\cdot\DGR[\Amn] \label{defdg1}\\
   & ~ & \DGL[\Perm\Amn\Qm]\cdot(\Perm\Amn\Qm)\cdot\DGR[\Perm\Amn\Qm]
                      ~ = ~ \Perm\cdot\left\{\DGL[\Amn]\cdot\Amn\cdot\DGR[\Amn]\right\}\cdot\Qm \label{defdg2}\\
   & ~ & \DGL[\Dmu\Amn\Dmv] ~ = ~ \DGL[\Amn], \label{defdg3}\\
   & ~ & \DGR[\Dmu\Amn\Dmv] ~ = ~ \DGR[\Amn] \label{defdg4}
\end{eqnarray}
The function $\DG[\Amn]$ is defined to be the rank-1 matrix guaranteed by
Lemma~\ref{daex} 
\begin{equation}
   \DG[\Amn] \circ \Amn ~ \equiv ~ \DGL[\Amn]\cdot\Amn\cdot\DGR[\Amn]
\end{equation}
i.e.,
\begin{equation}
   \DG[\Amn] ~ = ~ \Diag[\DGL[\Amn]]\cdot\Diag[\DGR[\Amn]]^{\mbox{\tiny T}}
\end{equation}
\end{definition}

\begin{definition}
A matrix $\Amn$ is defined to be an {\em elemental-nonzero matrix} if and only if
it does not have any element equal to zero. 
\end{definition}

The following lemma uses the elementwise matrix functions $\Log[\cdot]$ and $\Exp[\cdot]$,
where $\Log[\Amn]$ represents the result of taking the logarithm of the magnitude
of each element of $\Amn$ and $\Exp[\Amn]$ represents the taking of the
exponential of every element of $\Amn$.

\begin{lemma}
\label{nonzerodef}
Existence of a general diagonal scale function according to Definition~\ref{defdg} for
arguments without zero elements is established
by instantiating $\DGL[\Amn]$ and $\DGR[\Amn]$ as
\begin{eqnarray}
   \DGL[\Amn] & = & \Diag[\xmc] \\
   \DGR[\Amn] & = & \Diag[\ync] \\
   ~ & \textnormal{for} & ~ \nonumber \\
   \xmc\cdot\ynr & = &  \DG[\Amn] ~ = ~ \Exp\left[ \Jm\Lmn\Jn - \Lmn\Jn - \Jm\Lmn  \right] 
\end{eqnarray}
where $\Lmn=\Log[\Amn]$ and $\Jm$ has all elements equal to $1/m$ and $\Jn$ has all 
elements equal to $1/n$.
\end{lemma}

\begin{proof}
First it must be shown that $\Exp\left[ \Jm\Lmn\Jn - \Lmn\Jn - \Jm\Lmn  \right]$ is a rank-1 matrix. 
This can be achieved by expanding as
\begin{eqnarray}
       \Jm\Lmn\Jn - \Lmn\Jn - \Jm\Lmn  & = &  
                       (\frac{1}{2}\Jm\Lmn\Jn - \Lmn\Jn) +  (\frac{1}{2}\Jm\Lmn\Jn - \Jm\Lmn)\\
      ~ & = &   (\frac{1}{2}\Jm\Lmn - \Lmn) \Jn + \Jm(\frac{1}{2}\Lmn\Jn - \Lmn)\\
      ~ & = &  \umc\onr + \omc\vnr 
\end{eqnarray}
where 
\begin{eqnarray}   
     \umc & = & \frac{1}{n}\left(\frac{1}{2}\Jm - \Imm\right)\Lmn \cdot \onc \\
     \vnr & = &  \omr \cdot \Lmn\left(\frac{1}{2}\Jn - \Imn\right)/m
\end{eqnarray}
and then noting that the elementwise exponential of  $\umc\onr + \omc\vnr$ is the strictly positive
rank-1 matrix $\Exp[\umc]\cdot\Exp[\vnr]$, i.e., $\DGL[\Amn]=\Diag[\Exp[\umc]]$ and
$\DGR[\Amn]=\Diag[\Exp[\vmc]]$, which confirms existence and strict positivity as required.
Eq.(\ref{defdg1}) can be established by observing that
\begin{equation}
   \DGL[\Dmp\Amn\Emp]\cdot(\Dmp\Amn\Emp)\cdot\DGR[\Dmp\Amn\Emp] ~ \equiv ~ \DG[\Dmp\Amn\Emp]\circ(\Dmp\Amn\Emp)
\end{equation}
and letting 
$\dmc=\Diag[\Dmp]$,
$\enc=\Diag[\Emp]$, $\umc=\Log[\dmc]$,
$\vnc=\Log[\emc]$, and $\Lmn=\Log[\Amn]$:
\begin{eqnarray}   
    \DG[\Dmp\Amn\Emp] & = & \DG[ ~ (\dmc\onr) \circ \Amn \circ (\omc\enr) ~ ] \\
    ~ & = &  \DG[~ \Exp[\umc\onr]\circ \Amn \circ \Exp[\omc\vnr] ~ ] \\
    ~ & = & \Exp[~ \Jm(\umc\onr + \Lmn+ \omc\vnr)\Jn\\
    ~ & ~ &  ~ \phantom{\Exp[} - (\umc\onr + \Lmn+ \omc\vnr)\Jn  \\
    ~ & ~ &  ~ \phantom{\Exp[}  - \Jm(\umc\onr + \Lmn+ \omc\vnr) ~]\\
    ~ & = & \Exp[~ (\Jm\cdot\umc\onr + \Jm\Lmn\Jn + \omc\vnr\cdot\Jn) \\
    ~ & ~ &  ~ \phantom{\Exp[} - (\umc\onr + \Lmn\Jn + \omc\vnr\cdot\Jn) \\
    ~ & ~ &  ~ \phantom{\Exp[}  -  (\Jm\cdot\umc\onr + \Jm\Lmn + \omc\vnr) ~] \\
    ~ & = & \Exp[~ (-\umc\onr) + (\Jm\Lmn\Jn-\Lmn\Jn-\Jm\Lmn) + (-\omc\vnr) ~] \\
    ~ & = & \Exp[-\umc\onr] ~ \circ ~ \DG[\Amn]  ~ \circ ~ \Exp[- \omc\vnr] \\
    ~ & = & \Dmpi\cdot\DG[\Amn]\cdot\Empi  \label{disei}
\end{eqnarray}
where the last step recognizes that $-\umc=\Log[\Diag[\Dmpi]]$ and
$-\vnc=\Log[\Diag[\Empi]]$. The identity of Eq.(\ref{defdg1}) can then be shown as:
\begin{eqnarray}
      \DGL[\Dmp\Amn\Emp]\cdot(\Dmp\Amn\Emp)\cdot\DGR[\Dmp\Amn\Emp] & = & \DG[\Dmp\Amn\Emp]\circ(\Dmp\Amn\Emp) \\
      ~ & = & (\Dmpi\cdot\DG[\Amn]\cdot\Empi) \circ (\Dmp\Amn\Emp) \\
      ~ & = & \DG[\Amn]\circ\Amn \\
      ~ & = & \DGL[\Amn]\cdot\Amn\cdot\DGR[\Amn]
\end{eqnarray}
Eq.(\ref{defdg2})  holds as the indexing of the rows and columns of 
$\DGL[\Amn]$ and $\DGR[\Amn]$ (and $\DG[\Amn]$) is the same as that of $\Amn$.
Eqs.(\ref{defdg3}) and~(\ref{defdg4}) hold directly because Lemma~\ref{nonzerodef} 
only involves functions of the
absolute values of the elements of the argument matrix $\Amn$.
\end{proof}

\begin{theorem}
\label{ginvtp}
For an elemental-nonzero $m\times n$ matrix $\Amn$, the operator 
\begin{equation}
   \Aginv ~ \doteq ~ \DGR[\Amn]\cdot
         \pinv{\left(\DGL[\Amn]\cdot\Amn\cdot\DGR[\Amn]\right)}\cdot\DGL[\Amn]
         \label{ginvdef}
\end{equation}
satisfies for any nonsingular diagonal matrices $\Dm$ and $\Em$:
\begin{eqnarray}
   \Amn\Aginv\Amn & = & \Amn, \\
   \Aginv\Amn\Aginv & = & \Aginv, \\
   \ginv{(\Dm\Amn\Em)} & = & \Emi\Aginv\Dmi, \label{ginvuc}\\
   \mbox{rank}[\Aginv] & = & \mbox{rank}[\Amn] 
\end{eqnarray}
and is therefore a general unit-consistent generalized inverse.
\end{theorem}

\begin{proof}
The first two generalized inverse properties can be established from the corresponding properties
of the MP-inverse as: 
\begin{eqnarray}
\Amn\Aginv\Amn & = & \Amn\cdot \left\{\DGR[\Amn]\cdot
          \pinv{\left(\DGL[\Amn]\cdot\Amn\cdot\DGR[\Amn]\right)} 
          \cdot\DGL[\Amn]\right\} \cdot \Amn \\
~ & = & (\inv{\DGL[\Amn]}\cdot\DGL[\Amn])\cdot \\
~ & ~ & \phantom{AAA} \Amn\cdot\left\{\DGR[\Amn]
                         \cdot\pinv{\left(\DGL[\Amn]\cdot\Amn\cdot\DGR[\Amn]\right)}
                         \cdot\DGL[\Amn]\right\}\cdot\Amn \\
 ~ & ~ &  \phantom{AAAAAA} \cdot(\DGR[\Amn]\cdot\inv{\DGR[\Amn]}) \\
~ & = & \inv{\DGL[\Amn]}\cdot \\
 ~ & ~ &  \phantom{A} \underline{(\DGL[\Amn]\cdot\Amn\cdot\DGR[\Amn])
          \cdot\pinv{\left(\DGL[\Amn]\cdot\Amn\cdot\DGR[\Amn]\right)}\cdot
          (\DGL[\Amn]\cdot\Amn\cdot\DGR[\Amn])} \\
 ~ & ~ &  \phantom{AA} \cdot\inv{\DGR[\Amn]} \\
~ & = & \inv{\DGL[\Amn]}\cdot(\DGL[\Amn]\cdot\Amn\cdot\DGR[\Amn])
          \cdot\inv{\DGR[\Amn]} \\ 
~ & = & \Amn
\end{eqnarray}
and
\begin{eqnarray}
\Aginv\Amn\Aginv & = & \left\{\DGR[\Amn]\cdot
          \pinv{\left(\DGL[\Amn]\cdot\Amn\cdot\DGR[\Amn]\right)} 
          \cdot\DGL[\Amn]\right\}\cdot\Amn \\
 ~ & ~ &  ~ ~ ~ ~ ~ ~\cdot \left\{\DGR[\Amn]
          \cdot\pinv{\left(\DGL[\Amn]\cdot\Amn\cdot\DGR[\Amn]\right)}
          \cdot\DGL[\Amn]\right\} \\
~ & = & \DGR[\Amn]\\
 ~ & ~ &  \phantom{A}  \cdot \pinv{\left(\DGL[\Amn]\cdot\Amn\cdot\DGR[\Amn]\right)}\cdot 
          \underline{\left(\DGL[\Amn]\cdot\Amn\cdot\DGR[\Amn]\right)
          \cdot\pinv{\left(\DGL[\Amn]\cdot\Amn\cdot\DGR[\Amn]\right)}} \\
  ~ & ~ &  \phantom{AA} \cdot\DGL[\Amn] \\
~ & = & \DGR[\Amn]\cdot
        \pinv{\left(\DGL[\Amn]\cdot\Amn\cdot\DGR[\Amn]\right)}
        \cdot\DGL[\Amn] \\
~ & = & \Aginv
\end{eqnarray}

The general UC condition $\ginv{(\Dm\Amn\Em)}=\Emi  \Aginv \Dmi$, for any nonsingular diagonal
matrix $\Dm$, can be established using a polar decompositions $\Dm=\Dmp\Dmu$ and $\Em=\Emp\Emu$:
\begin{eqnarray}
\ginv{(\Dm\Amn\Em)} & = &  \DGR[\Dm\Amn\Em]\cdot
         \pinv{\left(\DGL[\Dm\Amn\Em]\cdot(\Dm\Amn\Em)\cdot\DGR[\Dm\Amn\Em]\right)}\cdot\DGL[\Dm\Amn\Em] \\
~ & = & \DGR[\Dmp\Amn\Emp]\cdot
         \pinv{\left(\DGL[\Dmp\Amn\Emp]\cdot(\Dm\Amn\Em)\cdot\DGR[\Dmp\Amn\Emp]\right)}\cdot\DGL[\Dmp\Amn\Emp] \\
~ & = & \Empi\cdot\DGR[\Amn]\cdot 
              \pinv{\left(\DGL[\Amn]\cdot\Dmpi\cdot(\Dm\Amn\Em)\cdot\Empi\cdot\DGR[\Amn]\right)}
              \cdot\DGL[\Amn]\cdot\Dmpi \\
~ & = &  \Empi\cdot\DGR[\Amn]\cdot 
              \pinv{\left(\DGL[\Amn]\cdot\Dmu\cdot\Amn\cdot\Emu\cdot\DGR[\Amn]\right)}
              \cdot\DGL[\Amn]\cdot\Dmpi     \label{gpinvua}  \\
~ & = & \Empi\cdot\DGR[\Amn]\cdot\Emui\cdot
              \pinv{\left(\DGL[\Amn]\cdot\Amn\cdot\DGR[\Amn]\right)}
              \cdot\Dmui\cdot\DGL[\Amn]\cdot\Dmpi    \label{gpinvub}   \\
~ & = & (\Empi\cdot\Emui)\cdot\DGR[\Amn]\cdot
              \pinv{\left(\DGL[\Amn]\cdot\Amn\cdot\DGR[\Amn]\right)}
              \cdot\DGL[\Amn]\cdot(\Dmui\cdot\Dmpi) \\
~ & = & \Emi\cdot\left\{\DGR[\Amn]\cdot
              \pinv{\left(\DGL[\Amn]\cdot\Amn\cdot\DGR[\Amn]\right)}
              \cdot\DGL[\Amn]\right\}\cdot\Dmi\\
~ & = & \Emi\cdot\Aginv\cdot\Dmi .
\end{eqnarray}
The rank-consistency condition of the theorem holds exactly as 
for the proof of Theorem~\ref{linvt}.
\end{proof}

The elemental-nonzero condition of Lemma~\ref{nonzerodef} is required
to ensure the existence of the elemental logarithms for $\Lmn=\Log[\Amn]$, so
the closed-form solution for the general unit-consistent matrix
inverse of Theorem~\ref{ginvtp} is applicable only to matrices without
zero elements. In many contexts involving general matrices there
is no reason to expect any elements to be identically zero, but in
some applications, e.g., compressive sensing, zeros are structurally
enforced. Unfortunately, Lemma~\ref{nonzerodef} cannot be extended
to accommodate zeros by a simple limiting strategy; however, 
results from matrix scaling theory can be applied to derive an 
unrestricted solution.

\section{The Fully General Unit-Consistent Generalized Inverse}
\label{gensec}

Given a nonnegative matrix $\Amn\in\mathbb{R}^{m\times n}$ with 
full support, $m$ positive numbers $S_1...S_m$, and $n$ positive 
numbers $T_1...T_n$, Rothblum \& Zenios~\cite{rz92} investigated the problem
of identifying positive diagonal matrices $\Um\in\mathbb{R}^{m\times m}$ and 
$\Vm\in\mathbb{R}^{n\times n}$ such that the product of the 
nonzero elements of each row $i$ of $\Amn'=\Um\Amn\Vm$ is $S_i$ and
the product of the nonzero elements of each column $j$ of 
$\Amn'$ is $T_j$. They provided an efficient solution, referred
to in their paper as {\em Program II}, and analyzed 
its properties. Specifically, for vectors 
$\mu\in\mathbb{R}^m$ and $\eta\in\mathbb{R}^n$, defined
in their paper\footnote{As will be seen, the precise definitions of
$\mu$ and $\eta$ will prove irrelevant for purposes of this paper.}, 
they proved the following:

\begin{theorem}
\label{rztheorem} (Rothblum \& Zenios\footnote{This theorem combines results from
theorems~4.2 and~4.3 of~\cite{rz92}. The variables $\Um$ and $\Vm$ are used 
here for positive diagonal matrices purely for consistency with that paper despite 
their exclusive use elsewhere in this paper to refer to unitary matrices. Although 
not stated explicitly by Rothblum and Zenios, Program~II
is easily verified to be permutation consistent.})
The following are equivalent:
\begin{enumerate}
   \item Program II is feasible.
   \item Program II has an optimal solution.
   \item $\prod_{i=1}^{m}(S_i)^{\mu_i} ~ = ~ \prod_{j=1}^{n}(T_j)^{\eta_j}$.
\end{enumerate}
If a solution exists then the matrix $\Amn'=\Um\Amn\Vm$ is the unique positive
diagonal scaling of $\Amn$ for which the product of the nonzero elements of each row
$i$ is $S_i$ and the product of the nonzero elements of each column $j$ is $T_j$.
\end{theorem}

Although $\Amn'$ is unique in Theorem~\ref{rztheorem}, the diagonal scaling 
matrices $\Um$ and $\Vm$ may not be. The implications of this, and the question
of existence, are addressed by the following theorem.  

\begin{theorem}
\label{rzcor}
For any nonnegative matrix $\Amn\in\mathbb{R}^{m\times n}$ with full support
there exist positive diagonal matrices $\Um\in\mathbb{R}^{m\times m}$ 
and $\Vm\in\mathbb{R}^{n\times n}$ such that the product of the 
nonzero elements of each row $i$ and column $j$ of $\Xm=\Um\Amn\Vm$ is $1$,
and $\Xm=\Um\Amn\Vm$ is the unique positive diagonal scaling of $\Amn$ which
has this property. Furthermore, if there do exist distinct positive diagonal 
matrices $\Um_1$, $\Vm_1$, $\Um_2$, and $\Vm_2$ such that 
\begin{equation}
\Xm ~=~ \Um_1\Amn\Vm_1 ~=~ \Um_2\Amn\Vm_2
\end{equation}
then $\inv{\Vm}_1\Apinv\inv{\Um}_1$ = $\inv{\Vm}_2\Apinv\inv{\Um}_2$.
\end{theorem}

\begin{proof}
The existence (and uniqueness) of a solution for Program~II according to 
Theorem~\ref{rztheorem} is equivalent to
\begin{equation}
   \prod_{i=1}^{m}(S_i)^{\mu_i} ~ = ~ \prod_{j=1}^{n}(T_j)^{\eta_j}
\end{equation}
which holds unconditionally, i.e., independent of $\mu$ and $\eta$, for
the case in which every $S_i$ and $T_j$ is~1. Proof of the {\em Furthermore} 
statement is given in Appendix~\ref{facta}.
\end{proof}

\begin{lemma}
\label{gendef}
Given an $m\times n$ matrix $\Amn$, let $\Xm$ be the matrix formed by removing
every row and column of $\Abs[\Amn]$ for which all elements are equal to zero, and define
$r[i]$ to be the row of $\Xm$ corresponding to row $i$ of $\Amn$ and $c[j]$
to be the column of $\Xm$ corresponding to column $j$ of $\Amn$. Let $\Um$ and
$\Vm$ be the diagonal matrices guaranteed to exist from the application of Program~II 
to $\Xm$ according to Theorem~\ref{rzcor}. 
Existence of a general-diagonal scale function according to Definition~\ref{defdg} for
$\Amn$ is provided by instantiating $\DGL[\Amn]=\Dm$ and $\DGR[\Amn]=\Em$ where
\begin{eqnarray}
   \Dm(i,i) & = &  
   \begin{cases}
              \Um(r[i],r[i])  & \textnormal{row $i$ of $\Amn$ is not zero}\\
             1 & \textnormal{otherwise}
   \end{cases}
\\
~ & ~ & \nonumber
\\
   \Em(j,j) & = &  
   \begin{cases}
              \Vm(c[j],c[j])  & \textnormal{column $j$ of $\Amn$ is not zero}\\
             1 & \textnormal{otherwise}
   \end{cases}
\end{eqnarray}
\end{lemma}

\begin{proof}
In the case that $\Amn$ has full support so that $\Xm=\Abs[\Amn]$ then
Theorem~\ref{rztheorem} guarantees that 
$\DGL[\Xm]\cdot\Xm\cdot\DGR[\Xm]$ is the unique diagonal scaling of 
$\Xm$ such that the product of the nonzero elements of each row and
column is 1. Therefore, the scale-invariance condition of Eq.(\ref{defdg1}):
\begin{equation}
    \DGL[\Dmp\Xm\Emp]\cdot(\Dmp\Xm\Emp)\cdot\DGR[\Dmp\Xm\Emp] 
            ~ = ~ \DGL[\Xm]\cdot\Xm\cdot\DGR[\Xm]
\end{equation}
holds for any positive diagonals $\Dmp$ and $\Emp$
as required. For the case of general $\Amn$ the construction 
defined by Lemma~\ref{gendef} preserves uniqueness with respect to nonzero 
rows and columns of $\DGL[\Amn]\cdot\Abs[\Amn]\cdot\DGR[\Amn]$,
i.e., those which correspond to the rows and columns of $\Um\Xm\Vm$, 
by the guarantee of Theorem~\ref{rztheorem}, and any row or column with all 
elements equal to zero is inherently scale-invariant, so  Eq.(\ref{defdg1}) holds
unconditionally for the construction defined by Lemma~\ref{gendef}.
The remaining conditions (permutation consistency and invariance
with respect to unitary diagonals) hold equivalently to the proof
of Lemma~\ref{nonzerodef}.
\end{proof}

At this point it is possible to establish the existence of a fully-general,
unit-consistent, generalized matrix inverse.

\begin{theorem}
\label{genginv}
For an $m\times n$ matrix $\Amn$ there exists an operator 
\begin{equation}
   \Aginv ~ \doteq ~ \DGR[\Amn]\cdot
         \pinv{\left(\DGL[\Amn]\cdot\Amn\cdot\DGR[\Amn]\right)}\cdot\DGL[\Amn]
\end{equation}
which satisfies for any nonsingular diagonal matrices $\Dm$ and $\Em$:
\begin{eqnarray}
   \Amn\Aginv\Amn & = & \Amn, \label{g1}\\
   \Aginv\Amn\Aginv & = & \Aginv, \label{g2}\\
   \ginv{(\Dm\Amn\Em)} & = & \Emi\Aginv\Dmi, \label{g3}\\
   \mbox{rank}[\Aginv] & = & \mbox{rank}[\Amn] \label{g5}.
\end{eqnarray}
\end{theorem}

\begin{proof}
The proof of Theorem~\ref{ginvtp} applies unchanged to
Theorem~\ref{genginv} except that the
elemental-nonzero condition imposed by Lemma~\ref{nonzerodef}
is removed by use of Lemma~\ref{gendef}.
\end{proof}

For completeness, the example of Eqs.(\ref{beginexample}-\ref{endexample})
with
\begin{equation}
\Dm ~=~ \left[\begin{array}{cc}
                         1 & 0\\ 0 & 2 \end{array}\right] ~ ~ ~ ~
\Amn ~=~ \left[\begin{array}{cc}
                         1/2 & -1/2\\ 1/2 & -1/2 \end{array}\right] ~ ~ ~ ~
\end{equation}
can be revisted to verify that 
\begin{equation}
\ginv{(\Dm\Amn\Dmi)} ~=~ \Dm\Aginv\Dmi ~=~ 
                 \left[\begin{array}{cc}
                         1/2 & ~1/4\\ -1 & -1/2 \end{array}\right] 
\end{equation}
where $\ginv{(\Dm\Amn\Dmi)}=\Dm\Aginv\Dmi$ as
expected. Extending the example with
\begin{equation}
\Em ~=~ \left[\begin{array}{cc}
                         5 & ~0\\ 0 & -3 \end{array}\right]
\end{equation}
it can be verified that 
\begin{equation}
\ginv{(\Dm\Amn\Em)} ~=~ \Emi\Aginv\Dmi ~=~ 
                 \left[\begin{array}{cc}
                         1/10 & 1/20\\ 1/6 & 1/12 \end{array}\right] 
\end{equation}
with equality as expected.

In practice the state space of interest may comprise subsets of variables
having different assumed relationships. For example, assume that 
$m$ state variables have incommensurate units while the remaining $n$
state variables are defined in a common Euclidean space, i.e., their relationship
should be preserved under orthonormal transformations. This assumption
requires that a linear transformation $\Amn$ should be consistent 
with respect to state-space transformations of the form
\begin{equation}
   \TR ~=~
   \left[\begin{array}{cc}
            \Dm   & {\bf 0} \\  
         {\bf 0} &    \Rm
   \end{array}\right]
\end{equation}
where $\Dm$ is a nonsingular $m\times m$ diagonal matrix and $\Rm$ is an
$n\times n$ orthonormal matrix. Thus the inverse of $\Amn$ cannot be 
obtained by applying 
either the UC inverse or the Moore-Penrose inverse, and the two inverses 
cannot be applied separately to distinct subsets of the state variables because 
all of the variables mix under the transformation. This can be seen from a
block-partition:
\begin{eqnarray}
   \begin{array}{rcl}
      \Amn & = &
         \left[ \begin{array}{cc} {\bf W} & {\bf X}\\
                             {\bf Y}& {\bf Z}\end{array} \right]
                       \begin{array}{l} {\bf \}}~m \\{\bf \}}~n \end{array}\\
                        & &  
                                \begin{array}{c}                                                          
                                   \;\underbrace{\;}_m\;\underbrace{\;}_n
                                \end{array}
   \end{array}
\end{eqnarray}
and noting that consistency in this case requires a generalized inverse that satisfies:
\begin{equation}
   \label{trblk}
   \tinv{(\TR_1\cdot\Amn\cdot\TR_2)} ~=~
         \tinv{\left[ 
                  \begin{array}{cc} \Dm_1{\bf W}\Dm_2 & \Dm_1{\bf X}\Rm_2\\
                                                 \Rm_1{\bf Y\Dm_2}& \Rm_1{\bf Z}\Rm_2
                  \end{array} \right]}
        ~=~ \inv{\TR_2}\cdot\Atinv\cdot\inv{\TR_1}~. 
\end{equation}
In the case of nonsingular $\Amn$ the partitioned inverse is unique:
\begin{equation}
   \inv{\Amn} ~=~
   \left[
     \begin{array}{cc}
     ({\bf W}-{\bf X}{\bf Z}^{-1}{\bf Y})^{-1} & -{\bf W}^{-1}{\bf X}({\bf Z}-{\bf Y}{\bf W}^{-1}{\bf X})^{-1} \\
      -{\bf Z}^{-1}{\bf Y}({\bf W}-{\bf X}{\bf Z}^{-1}{\bf Y})^{-1} & ({\bf Z}-{\bf Y}{\bf W}^{-1}{\bf X})^{-1}
     \end{array}
    \right]
\end{equation}
and is unconditionally consistent. Respecting the block constraints implicit from Eq.(\ref{trblk}), the desired
generalized inverse for singular $\Amn$ under the present assumptions can be verified as:
\begin{equation}
   \Atinv ~=~
   \left[
     \begin{array}{cc}
     \ginv{({\bf W}-{\bf X}\pinv{{\bf Z}}{\bf Y})} & -\ginv{{\bf W}}{\bf X}\pinv{({\bf Z}-{\bf Y}\ginv{{\bf W}}{\bf X})} \\
      -\pinv{{\bf Z}}{\bf Y}\ginv{({\bf W}-{\bf X}\pinv{{\bf Z}}{\bf Y})} & \pinv{({\bf Z}-{\bf Y}\ginv{{\bf W}}{\bf X})}
     \end{array}
    \right] .
\end{equation}
The general case involving different assumptions for more than two subsets of state variables 
(possibly different for the left and right spaces of the transformation) can be solved analogously
with appropriate partitioning.

The generalized inverse of Theorem~\ref{genginv} is unique when instantiated
using the construction defined by Lemma~\ref{gendef} by virtue of
the uniqueness of both the Moore-Penrose inverse and the scaling of
Theorem~\ref{rztheorem} (alternative scalings are discussed in
Appendix~\ref{altsca}), and it completes a trilogy of generalized matrix inverses 
that exhausts the standard family of transformation invariants. Specifically,
the Drazin inverse is consistent with respect to similarity transformations,
the Moore-Penrose inverse is consistent with respect to unitary/orthonormal
transformations, and the new generalized inverse is consistent with respect to 
diagonal transformations. 

In the next section it is demonstrated that the general approach for 
obtaining the UC inverse can be efficiently 
applied to a wide variety of other matrix decompositions and
operators (including other generalized matrix inverses) to impose
unit consistency.

\section{Unit-Consistent/Invariant Matrix Decompositions}
\label{ucsvd}

Unit consistency has been suggested in the past
as a critical consideration in specific applications (e.g.,
robotics~\cite{duffy90,doty2} and
data fusion~\cite{jku95}), but the means for
enforcing it have been limited because the most
commonly applied tools in linear systems analysis,
the eigen and singular-value decompositions, are inherently
not unit consistent and therefore require
UC alternatives. This motivates the need to extend
unit-consistency to other areas of matrix analysis.
This clearly includes transformations $T[\Amn]$ which can be 
redefined in UC form as 
\begin{equation} 
   \inv{\DGL[\Amn]}\cdot 
   T\left[ \DGL[\Amn]\cdot\Amn\cdot\DGR[\Amn]\right]
   \cdot\inv{\DGR[\Amn]}
\end{equation}
and functions $f[\Amn]$ which can be redefined in unit 
scale-invariant form as
$f[\DGL[\Amn]\cdot\Amn\cdot\DGR[\Amn]]$, but it also
extends to matrix decompositions. 

The Singular Value Decomposition (SVD) is among the most
powerful and versatile tools in linear algebra and 
data analytics~\cite{ytt11,michelena93,ltp10,abb00}.
The Moore-Penrose generalized inverse of $\Amn$ can be obtained 
from the SVD of $\Amn$
\begin{equation}
      \Amn ~=~ \Um\Sm\Vmi
\end{equation}
as
\begin{equation}
      \Apinv ~=~ \Vm\tinv{\Sm}\Umi
\end{equation}
where $\Um$ and $\Vm$ are unitary, $\Sm$ is the diagonal matrix
of singular values of $\Amn$, and $\tinv{\Sm}$ is the matrix obtained from
inverting the nonzero elements of $\Sm$. This motivates the
following definition.

\begin{definition}
The Unit-Invariant Singular-Value Decomposition
(UI-SVD) is defined as
\begin{equation}
   \Amn ~=~ \Dm\cdot\Um\Sm\Vmi\cdot\Em
\end{equation}
with $\Dm=\inv{\DGL[\Amn]}$, $\Em=\inv{\DGR[\Amn]}$, and
$\Um\Sm\Vmi$ is the SVD of $\Xm=\DGL[\Amn]\cdot\Amn\cdot\DGR[\Amn]$.
The diagonal elements of $\Sm$ are referred to as the 
{\em unit-invariant} (UI) {\em singular values} of $\Amn$. 
\end{definition}

Given the UI-SVD of a matrix $\Amn$
\begin{equation}
      \Amn ~=~ \Dm\cdot\Um\Sm\Vmi\cdot\Em
\end{equation}
the UC generalized inverse of $\Amn$
can be expressed as
\begin{equation}
     \Aginv ~=~ \Emi\cdot\Vm\tinv{\Sm}\Umi\cdot\Dmi.
\end{equation}
Unlike the singular
values of $\Amn$, which are invariant with respect
to arbitrary left and right unitary transformations of
$\Amn$, the UI singular values
are invariant with respect to arbitrary left and right nonsingular
diagonal transformations\footnote{A {\em left} 
UI-SVD can be similarly defined as $\Amn  =  \Dm\cdot\Um\Sm\Vmi$
with $\Dm=\inv{\DL[\Amn]}$ and 
$\Um\Sm\Vmi$ being the SVD of $\Xm=\DL[\Amn]\cdot\Amn$.
The resulting {\em left unit-invariant singular values} are
invariant with respect to nonsingular left diagonal transformations
and right unitary transformations of $\Amn$ (the latter property
is what motivated the right unitary-invariance requirement
of Definition~\ref{defdl} and the consequent use of a
unitary-invariant norm in the construction of
Lemma~\ref{dconst}.). A {\em right}
UI-SVD can be defined analogously to interchange the
left and right invariants.}. Thus, functions of the unit-invariant
singular values are unit-invariant with respect to $\Amn$.

The largest $k$ singular values of a matrix 
(e.g., representing a photograph, video sequence,
or other object of interest) can be used to
define a unitary-invariant 
signature~\cite{czt12,gjmeyer03,hong91,luochen94,jeonglee06,jeongle09}
which supports computationally efficient similarity testing.
However, many sources of error in practical applications
are not unitary. As a concrete example, 
consider a system in which
a passport or driving license is scanned to produce
a rectilinearly-aligned and scaled image that is to 
be used as a key to search an existing image database.
The signature formed from the largest $k$
unit-invariant singular values can be used for
this purpose to provide robustness to amplitude
variations among the rows and/or columns of the
image due to the scanning process.

The UI-SVD may also offer advantages
as an alternative to the conventional 
SVD, or truncated SVD, used by existing methods
for image and signal 
processing, cryptography, digital watermarking, 
tomography, and other applications in order to provide  
state-space or coordinate-aligned robustness to 
noise\footnote{Such applications are examined in
greater detail in a longer preprint from which this
paper was derived~\cite{uarxiv}. In applications in which the signature 
provided by the set of UI singular values may be too 
concise~\cite{ttw03}, e.g., because it is permutation 
invariant, a vectorization of the matrix 
$\Amn\circ(\Aginv)^{\mbox{\tiny T}}$
can be used as a more discriminating UI
signature.}.

More generally, the approach used to define
unit scale-invariant singular values can be applied
to other matrix decompositions, though the invariance
properties may be different. In the case of 
scale-invariant eigenvalues for square $\Amn$, i.e.,
eig[$\DGL[\Amn]\cdot\Amn\cdot\DGR[\Amn]$], the
invariance is limited to diagonal transformations
$\Dm\Amn\Em$ such that $\Dm\Em$ is
nonnegative real, e.g., $\Dmp\Amn\Emp$, 
$\Dm\Amn\Dm$ (or $\Dm\Amn\bar{\Dm}$ for
complex $\Dm$), and $\Dm\Amn\Dmi$. In applications
in which changes of units can be assumed to take the
form of positive diagonal transformations, the 
scale-invariant (SI) eigenvalues can therefore be 
taken as a complementary or alternative signature 
to that provided by the UI singular values.

\section{Discussion}
\label{discsec}

The principal contribution of this paper is the derivation
of a unit-consistent generalized matrix inverse. Its 
consistency with respect to diagonal transformations
provides an alternative to the Drazin inverse, which is
only consistent with respect to similarity transformations, 
and the Moore-Penrose inverse, which is only consistent
with respect to unitary/orthonormal transformations.
This inverse has broad potential applications as a replacement
to Moore-Penrose, e.g., for unit-consistent or unit-invariant
gradient descent for optimization and deep learning.
Another contribution of the paper is a demonstration
that a partitioned inverse can be constructed to respect
different consistency assumptions associated with 
different subsets of variables of a defined state space.
This is essential to permit a large, complex linear system
to be mathematically expressed and analyzed in a unified 
manner. More generally, different inverses can also be
applied in parallel to obtain sets of distinct solutions, e.g., 
in order to extract information under different invariance
and/or consistency assumptions in pattern
recognition and machine learning applications. 

It has been emphasized that the appropriate choice 
of inverse depends critically on the assumed properties
of the system of interest.
It should also be emphasized that while the Drazin, Moore-Penrose,
and UC inverses may be {\em analytically} special because of their
respective algebraic properties, there are many practical
situations in which a non-algebraic ``heuristic'' inverse may
represent the more appropriate choice. For example, 
consider the following singular matrix:
\begin{equation}
\Amn ~=~ \left[\begin{array}{ccc}
                         \frac{1}{2} & 0  & 0\\ 0 & 3 & 0 \\ 0 & 0 & 0 \end{array}\right] 
\end{equation}
From a practical perspective, if the matrix is assumed to be nonnegative
then its ``true'' inverse is
\begin{equation}
\inv{\Amn}~=~ \left[\begin{array}{ccc}
                         2 & 0  & 0\\ 0 & \frac{1}{3} & 0 \\ 0 & 0 & \infty \end{array}\right] 
\end{equation}
and could be approximated with a large number in place of infinity as
\begin{equation}
\label{blowup}
\Atinv ~=~ \left[\begin{array}{ccc}
                         2 & 0  & 0\\ 0 & \frac{1}{3} & 0 \\ 0 & 0 & 10^{15} \end{array}\right] 
\end{equation}
so as to capture the blow-up of zero when inverted.
This blow-up behavior is lost completely if an algebraic inverse is
applied:
\begin{equation}
\Adinv ~=~ \Apinv ~=~ \Aginv ~=~ \left[\begin{array}{ccc}
                         2 & 0  & 0\\ 0 & \frac{1}{3} & 0 \\ 0 & 0 & 0 \end{array}\right] 
\end{equation}
This counter-intuitive (and in many practical situations {\em wrong}) result derives from the fact
that no finite value can be chosen in a way that preserves any form of pure {\em algebraic}
consistency\footnote{It might not be inappropriate to formally define the result of
Eq.(\ref{blowup}) as a {\em blow-up consistent} generalized matrix inverse.}. 
This is just one more example of the danger of  blindly
using the Moore-Penrose pseudoinverse (or any other generalized inverse) to deal 
generically with the problem of inverting a singular matrix. In summary, it is essential
to make the choice of generalized matrix inverse based on the application-specific 
properties that are necessary to preserve, and this paper has provided the appropriate
choice for unit consistency.

\appendices

\section{Uniqueness of the UC Inverse}
\label{facta}

By virtue of the uniqueness of the Moore-Penrose inverse, the UC inverse from 
Theorem~\ref{genginv} is uniquely determined given a scaling $\Amn=\DGL\Xm\DGR$ 
produced according to Theorem~\ref{rzcor}. However, the positive diagonal matrices $\DGL[\Amn]$ 
and $\DGL[\Amn]$ are not necessarily unique, so there may exist distinct positive
diagonal matrices $\Dma$ and $\Dmb$ and $\Ema$ and $\Emb$ such that
\begin{equation}
     \Amn ~=~ \Dma\Xm\Ema ~=~ \Dmb\Xm\Emb. 
\end{equation}
What remains is to establish the uniqueness of $\Aginv$ in this case, i.e., that
\begin{equation}
    \Dma\Xm\Ema = \Dmb\Xm\Emb ~\implies~
   \Emia\Xpinv\Dmia = \Emib\Xpinv\Dmib.
\end{equation}

We begin by noting that if an arbitrary $m\times n$ matrix $\Amn$ has rank $r$ then it can be factored~\cite{big} as the product
of an $m\times r$ matrix $\Fm$ and an $r\times n$ matrix
$\Gm$ as
\begin{equation}
\Amn ~ = ~ \Fm\Gm
\end{equation}
The Moore-Penrose inverse can then be expressed in terms of
this rank factorization as
\begin{equation}
\Apinv ~ = ~ \Gmt\cdot\inv{(\Fmt\cdot\Amn\cdot\Gmt)}\cdot\Fmt \label{rfac}
\end{equation}
where $\Gmt$ and $\Fmt$ are the conjugate transposes
of $\Gm$ and $\Fm$. 

Because $\Dma\Xm\Ema=\Dmb\Xm\Emb$ implies
\begin{equation}
     \Xm ~=~ \Dmib\Dma\Xm\Ema\Emib 
\end{equation}
then from the rank factorization $\Xm=\Fm\Gm$ 
we can obtain an alternative factorization
\begin{equation}
     \Xm ~=~ \Fm'\Gm' ~=~ (\Dmib\Dma\Fm)(\Gm\Ema\Emib)
\end{equation}
from the fact that the ranks of $\Fm$ and
$\Gm$ are unaffected by nonsingular diagonal
scalings. Applying the rank factorization identity
for the Moore-Penrose inverse then yields
\begin{eqnarray}
     \Xpinv & = & \pinv{(\Fm'\Gm')} \\
~ & = & (\Gm\Ema\Emib)^*\cdot
           \inv{ \left( (\Dmib\Dma\Fm)^*\cdot\Xm\cdot (\Gm\Ema\Emib)^* \right) }
                        \cdot (\Dmib\Dma\Fm)^* \\
~ & = & \Ema\Emib\Gmt\cdot
                \inv{ \left( (\Fmt\Dmib\Dma)\Xm(\Ema\Emib\Gmt) \right) }
                        \cdot\Fmt\Dmib\Dma \\
~ & = & (\Ema\Emib)\cdot\Gmt\cdot
                \inv{ \left( \Fmt\cdot(\underline{\Dmib\Dma\Xm\Ema\Emib})\cdot\Gmt \right) }
                        \cdot\Fmt\cdot(\Dmib\Dma) \\
~ & = & (\Ema\Emib)\cdot\left(\Gmt\cdot
                \inv{ \left( \Fmt\Xm\Gmt \right) }\cdot
                        \Fmt\right)\cdot(\Dmib\Dma) \\
~ & = & \Ema\Emib\left(\underline{\Gmt\cdot\inv{(\Fmt\Xm\Gmt)}\cdot\Fmt}\right)\Dmib\Dma \\
~ & = & \Ema\Emib\Xpinv\Dmib\Dma 
\end{eqnarray}
which implies\footnote{Note that the diagonal matrices commute and are real so, e.g., 
$\Dm^*=\Dm$.} 
\begin{eqnarray}
\Emia\Xpinv\Dmia & = & \Emia\cdot\left(\Ema\Emib\Xpinv\Dmib\Dma\right)\cdot\Dmia \\
~ & = & (\Emia\Ema)\cdot\Emib\Xpinv\Dmib\cdot(\Dma\Dmia) \\
~ & = & \Emib\Xpinv\Dmib
\end{eqnarray}
and thus establishes that $\Emia\Xpinv\Dmia=\Emib\Xpinv\Dmib$ and therefore
that the UC generalized inverse $\Aginv$ is unique.

Using a similar but more involved application of rank factorization it can be 
shown that the UC generalized matrix inverse satisfies
\begin{equation}
   \Aginv \cdot \ginv{(\Aginv)} \cdot \Aginv ~=~ \Aginv
\end{equation}
which is weaker than the uniquely-special property of the Moore-Penrose
inverse:
\begin{equation}
    \pinv{(\Apinv)}=\Amn. 
\end{equation}

\section{Alternative Constructions}
\label{altsca}

The proofs of Theorems~\ref{linvt} and and~\ref{ginvtp}
(and consequently Theorem~\ref{genginv}) do not actually 
require the general unitary consistency property of 
the Moore-Penrose inverse and instead 
only require diagonal unitary consistency, e.g.,
in Eqs.(\ref{lmpua})-(\ref{lmpub}) as
\begin{equation}
    \pinv{(\Dmu\Amn)} ~ = ~ \Apinv\Dmui
\end{equation}
and in Eqs.(\ref{gpinvua})-(\ref{gpinvub}) as   
\begin{equation}
   \pinv{(\Dmu\Amn\Emu)} ~ = ~ \Emui\Apinv\Dmui
\end{equation}
for unitary diagonal matrices $\Dmu$ and $\Emu$.
Thus, the Moore-Penrose inverse could be replaced
with an alternative which maintains the other 
required properties but satisfies this weaker 
condition in place of general unitary consistency.

Similarly, the scalings defined by
Lemmas~\ref{nonzerodef} and~\ref{gendef} 
are not necessarily the only ones
that may be used to satisfy the conditions of 
Definition~\ref{defdg}. More specifically,
Lemmas~\ref{nonzerodef} and~\ref{gendef}
define left and right nonnegative diagonal scaling functions
$\DGL[\Amn]$ and $\DGR[\Amn]$ satisfying
\begin{equation}
   \DGL[\Amn]\cdot\Amn\cdot\DGR[\Amn]  ~ = ~
   \DGL[\Dmp\Amn\Emp]\cdot\Dmp\Amn\Emp\cdot\DGR[\Dmp\Amn\Emp]
\end{equation}
for all positive diagonals $\Dmp$ and $\Emp$. Because the unitary
factors of the elements of $\Amn$ are unaffected by the nonnegative
scaling, the scalings can be constructed without loss of generality from
$\Abs[\Amn]$. If nonnegative $\Amn$ is square, irreducible, and has full 
support then such a scaling can be obtained by alternately normalizing the rows
and columns to have unit sum using the Sinkhorn iteration~\cite{sink64,sink67}.
The requirement for irreducibility stems from the fact that the process
cannot always converge to a finite left and right scaling. For example,
the matrix 
\begin{equation}
   \label{trimat}
   \begin{bmatrix}
      a & b \\ 0 &  c
   \end{bmatrix}
\end{equation} 
cannot be scaled so that the rows and columns sum to unity unless
the off-diagonal element $b$ is driven to zero, which is not possible
for any finite scaling. In other words, the Sinkhorn unit-sum
condition cannot be jointly satisfied with respect to both the set
of row vectors and the set of column vectors. What is needed,
therefore, is a measure of vector ``size'' that can be applied
within a Sinkhorn-type iteration but is guaranteed to converge to
a finite scaling\footnote{This definition and the
subsequently-defined instance, $s_{a,b}[\uv]$, may be of 
independent interest for analyzing properties
of low-rank subspace embeddings in high-dimensional vector 
spaces, e.g., infinite-dimensional spaces.}. 

\begin{definition}
For all vectors $\uv$ with elements from a normed division algebra, 
a nonnegative composable size function $s[\uv]$ is defined
as satisfying the following conditions for all $\alpha$:
\begin{eqnarray}
   s[\uv] & = & 0  ~ ~ \Leftrightarrow ~ ~ \uv ~ = ~ {\bf 0}\\
   s[\alpha\uv] & = & | \alpha | \cdot s[\uv] \\
   s[\bv] & = & 1 ~ ~ ~ ~\forall \bv \in \{0,1\}^n -\znc \\
   s[\uv] & = &  s[\uv \otimes \bv] ~ = ~ s[\bv \otimes \uv] ~ ~ ~ ~\forall \bv \in \{0,1\}^n
\end{eqnarray}

\end{definition}
The defined size function provides a measure of scale that
is homogeneous, permutation-invariant, and invariant with respect to tensor expansions
involving identity and zero elements. More intuitively, however, $s[\uv]$ 
can be thought of as a ``mean-like'' measure taken over the magnitudes of the nonzero 
elements of $\uv$. With the imposed condition $s[\mbox{\bf 0}]\doteq 0$
the following instantiations can also be verified to satisfy the definition: 
\begin{eqnarray}
   \stimes[\uv] & \doteq & \left(\prod_{k\in S} |\uv_k| \right)^{1/|S|} ~ ~ ~ ~ ~ ~ ~ ~ ~  j\in S ~ ~ \mbox{iff} ~ ~ \uv(j)\neq 0 \\
   s_p[\uv] & \doteq & \norm{\uv}_p ~/~ |S|^{1/p}  ~ ~  ~ ~ ~ ~ ~ ~ ~ ~ ~ ~  j\in S ~ ~ \mbox{iff} ~ ~ \uv(j)\neq 0 \\
   s_{a,b}[\uv] & \doteq &  \left(\frac{\sum_i |\uv_i|^{a+b}}{\sum_i |\uv_i|^a}\right)^{1/b} ~ ~ ~ ~ ~ ~ ~ ~ ~ ~ ~ a>0,~b>0
\end{eqnarray}
The first case, $\stimes(\uv)$, is more easily interpreted as the geometric mean of the 
nonzero elements of $\uv$. Its application in a Sinkhorn-type iteration converges to a 
unique scaling in which the {\em product} of the nonzero elements in each row and column has unit
magnitude. If $a$, $b$, and $c$ are positive for the matrix 
of Eq.(\ref{trimat}) then the scaled result using $s_p[\uv]$ is
\begin{equation}
   \begin{bmatrix}
      1 & 1 \\ 0 &  1
   \end{bmatrix}
\end{equation} 
where the product of the nonzero elements in each row and column is unity and the
particular left and right diagonal scalings are determined by the values of 
$a$, $b$, and $c$. It can be shown that for all elemental nonzero 
matrices that the scaling produced 
using $\stimes(\uv)$ is equivalent to that produced by the 
constructions defined by Lemmas~\ref{nonzerodef} and~\ref{gendef}
and that the iteration is fast-converging.

The row/column conditions imposed by $s_p[\uv]$ can most easily be understood
in the case of $p=1$, for which it is equivalent to the mean of the absolute values of
the nonzero elements of $\uv$. In the case of $p=2$, if a vector $\vv$ is formed 
from the $m$  nonzero elements of $\uv$ then 
\begin{equation}
        s_2[\uv] ~ = ~  \norm{\vv}_2 ~/ ~m^{1/2}
\end{equation}
In the example of the $2\times 2$  matrix of Eq.(\ref{trimat}) the scaled result 
produced using $s_p[\uv]$ for any $p>0$ happens to be the same as that produced
using $\stimes[\uv]$. For nontrivial matrices, however, the results for different 
$p$ are not generally (nor typically) equivalent to each other or to that 
produced by $\stimes(\uv)$. 

The third size function, $s_{a,b}[\uv]$, satisfies the required conditions
without imposing special treatment of zero elements. In other words, it is a
continuous function of the elements of $\uv$ and would therefore 
appear to be a more natural choice for instantiating $\DGL[\Amn]$ and 
$\DGR[\Amn]$ for analysis purposes, e.g., in the limit as $a$ and $b$ 
go to zero where $s_{a,b}[\uv]\equiv\stimes[\uv]$. (It should be noted
that the homogeneity properties of $s_{a,b}[\uv]$ hold generally
for any $a$ and $b$ from a normed division algebra with $0^0\doteq 1$, 
and it subsumes $s_p[\uv]$ in the limiting cases $a\rightarrow 0$ and/or 
$b\rightarrow 0$.)

\section{Implementations}
\label{code}

Below are basic Octave/Matlab implementations
of some of the methods developed in the paper.
Although not coded for maximum efficiency or numerical robustness, they should 
be sufficient for experimental corroboration of theoretically-established properties.\\
~\\
\noindent The following function computes $\Aginv$ for 
$m\times n$ real or complex matrix $\Amn$. It has complexity
dominated by the Moore-Penrose inverse calculation,
which is $O(mn\cdot\min(m,n))$.
\begin{verbatim}
function Ai = uinv(A)
    [S, dl, dr] = dscale(A);
    Ai = pinv(S) .* (dl * dr)';
end
\end{verbatim}

\noindent The following function evaluates the 
UC/UI singular values of the real or
complex matrix $\Amn$.
\begin{verbatim}
function s = usvd(A)
    s = svd(dscale(A));
end
\end{verbatim}

\noindent The following function evaluates the 
UC/UI singular-value decomposition of the $m\times n$ real or 
complex matrix $\Amn$.
\begin{verbatim}
function [D, U, S, V, E] = usv_decomp(A)
    [S, dl, dr] = dscale(A);
    D = diag(1./dl);  E = diag(1./dr);
    [U, S, V] = svd(S);
end
\end{verbatim}

\noindent The following function computes the unique
positively-scaled matrix
$\Sm=\DGL[\Amn]\cdot\Amn\cdot\DGR[\Amn]$
with diagonal left and right scaling matrices
$\DGL[\Amn]=\mbox{diag[dl]}$ and
$\DGR[\Amn]=\mbox{diag[dr]}$. It has 
$O(mn)$ complexity for $m\times n$ real or 
complex matrix $\Amn$.
\begin{verbatim}
function [S, dl, dr] = dscale(A)
    tol = 1e-15;    
    [m, n] = size(A);
    L = zeros(m, n);    M = ones(m, n);   
    S = sign(A);   A = abs(A);
    idx = find(A > 0.0);  L(idx) = log(A(idx));
    idx = setdiff(1 : numel(A), idx);
    L(idx) = 0;    M(idx) = 0;   
    r = sum(M, 2);   c = sum(M, 1);   
    u = zeros(m, 1); v = zeros(1, n);
    dx = 2*tol;  
    while (dx > tol)
        idx = c > 0;
        p = sum(L(:, idx), 1) ./ c(idx);
        L(:, idx) = L(:, idx) - repmat(p, m, 1) .* M(:, idx);
        v(idx) = v(idx) - p;  dx = mean(abs(p));
        idx = r > 0;
        p = sum(L(idx, :), 2) ./ r(idx);
        L(idx, :) = L(idx, :) - repmat(p, 1, n) .* M(idx, :);
        u(idx) = u(idx) - p;  dx = dx + mean(abs(p));
    end    
    dl = exp(u);   dr = exp(v);
    S = S.* exp(L);
end
\end{verbatim}


\begin{thebibliography}{7}

\bibitem{abb00}
O. Alter, P.O. Brown, D. Botstein, 
``Singular Value Decomposition for Genome-Wide Expression Data Processing and Modeling,''
{\em Proc Natl Acad Sci}, 97(18):10101-6, 2000.

\bibitem{altergolub04}
O. Alter O, G.H. Golub, ``Integrative Analysis of Genome-Scale Data by Using Pseudoinverse Projection Predicts Novel Correlation Between DNA Replication and RNA Transcription,'' {\em Proc Natl Acad Sci}, 101(47):16577-16582, 2004.


\bibitem{big}
A. Ben-Israel and N.E. Greville, {\em Generalized Inverses: Theory
and Applications}, 2nd Edition, Springer-Verlag, 2003. 

\bibitem{berryman00}
J. G. Berryman, ``Analysis of Approximate Inverses in Tomography. I. Resolution analysis,'' {\em Optimization and Engineering}, 1, 87-117, 2000.

\bibitem{cmr76}
S.L. Campbell, C.D. Meyer, and N.J. Rose, ``Applications of the Drazin Inverse to Linear Systems of Differential Equations with Singular Constant Coefficients,'' {\em SIAM Journal of Applied Mathematics}, Vol. 31, No. 3, 1976.

\bibitem{czt12}
B. Cui, Z. Zhao, W.H. Tok, ``A Framework for Similarity Search of Time Series 
Cliques with Natural Relations,'' {\em IEEE Transaction on Data and Knowledge 
Engineering},  2012.


\bibitem{doty2}
K. L. Doty, C. Melchiorri, and C. Bonivento, ``A Theory of Generalized Inverses Applied to
Robotics,'' {\em International Journal of Robotics
Research}, vol. 12, no. 1, pp. 1-19, 1995.

\bibitem{drazin}
M. Drazin, ``Pseudo-Inverses in Associative Rings and Semigroups,'' 
{\em The American Mathematical Monthly}, 65:7, 1958.

\bibitem{duffy90}
J. Duffy, ``The Fallacy of Modern Hybrid Control Theory that is Based on
`Orthogonal Complements' of Twists and Wrenches Spaces'', 
{\em Int. J. of Robotic Systems}, 7(2), 1990.

\bibitem{hong91}
Z-Q Hong, ``Algebraic feature extraction of image for recognition,'' 
{\em Pattern Recognition}, 24(3), 211-219, 1991.

\bibitem{jeongle09}
KM Jeong and J-J Lee, ``Video Sequence Matching Using Normalized 
Dominant Singular Values,'' {\em Journal of the
Korea Multimedia Society}, Vol.12:12,  Page 785-793, 2009.

\bibitem{jeonglee06}
KM Jeong, J-J Lee, Y-H Ha, ``Video sequence matching using singular value decomposition,''
{\em Proc. 3rd Int. Conf. Image Analysis and Recognition} (ICIAR), pp 426-435, 2006.

\bibitem{ltp10}
F. Leblond, K.M. Tichauer, B.W. Pogue, ``Singular Value Decomposition Metrics
Show Limitations of Detector Design in Diffuse Fluorescence Tomography,''
{\em Biomedical Optics Express.}, 1(5):1514-1531, 2010.

\bibitem{luochen94}
J. H. Luo and C. C. Chen, ``Singular Value Decomposition for Texture Analysis,''
{\em Applications of Digital Image Processing XVII}, SPIE Proceedings,
vol. 2298, pp.407-418, 1994.

\bibitem{gjmeyer03}
G.J. Meyer, {\em Classification of Radar Targets using Invariant Features},
Dissertation, Air Force Institute of Technology, AFIT/DS/ENG/03-04, 2003.



\bibitem{michelena93}
R.J. Michelena, ``Singular Value Decomposition for Cross-Well Tomography,''
{\em Geophysics}, 58(11):1655-1661, 1993.





\bibitem{rz92}
U.G. Rothblum and S.A. Zenios, ``Scalings of Matrices Satisfying Line-Product Constraints and Generalizations,'' 
{\em Linear Algebra and Its Applications}, 175: 159-175, 1992.


\bibitem{sink64}
R. Sinkhorn, ``A relationship between arbitrary positive matrices and doubly stochastic matrices,'' 
{\em Ann. Math. Statist.}, 35, 876-879, 1964.

\bibitem{sink67}
R. Sinkhorn and P.  Knopp, ``Concerning nonnegative matrices and doubly stochastic matrices,'' 
{\em Pacific J. Math.}, 21, 343-348, 1967.



\bibitem{ttw03}
Y. Tian, T. Tan, and Y. Wang, ``Do singular values contain adequate information for face recognition?,'' {\em  Pattern Recognition}, 36:649-655, 2003.

\bibitem{uarxiv} J.K. Uhlmann, ``Unit Consistency, Generalized Inverses, and
Effective System Design Methods,'' arXiv:1604.08476v2 [cs.NA] 11 Jul 2017.

\bibitem{jku95}
J.K. Uhlmann, {\em Dynamic Map Building and Localization: New Theoretical Foundations},
pp. 86-87, Doctoral Dissertation, University of Oxford, 1995.

\bibitem{ytt11}
H. Yanai, K. Takeuchi, Y. Takane, {\em Projection Matrices, Generalized Inverse Matrices, and 
Singular Value Decomposition},'' Springer, ISBN-10:1441998861, 2011.


\end{thebibliography}
\end{document}